\documentclass[letterpaper, 10 pt, conference]{ieeeconf} 	

\IEEEoverridecommandlockouts    

\newtheorem{theorem}{Theorem}[section]
\newtheorem{lemma}[theorem]{Lemma}

\newtheorem{corollary}[theorem]{Corollary}

\newtheorem{remark}{Remark}
\newtheorem{assumption}{Assumption}

\newenvironment{pfof}[1]{\vspace{1ex}\noindent{\itshape 
	Proof of #1:}\hspace{0.5em}} {\hfill\QEDBL\vspace{1ex}}

\newcommand{\mc}{\mathcal}

\newcommand{\rank}{\operatorname{rank}}
\newcommand{\real}{\mathbb{R}} 
\newcommand{\integ}{\mathbb{Z}}

\newcommand{\realpos}{\mathbb{R}_{\geq 0}}
\newcommand{\integpos}{\mathbb{Z}_{\geq 0}}

\newcommand{\tsp}{\mathsf{T}} 
 
\newcommand{\inv}{{\negat 1}} 
\newcommand{\negat}{\scalebox{0.75}[.9]{\( - \)}}

\newcommand*{\QEDB}{\hfill\ensuremath{\square}}
\newcommand*{\QEDBL}{\hfill\ensuremath{\blacksquare}}
\newcommand\oprocendsymbol{\hbox{$\square$}}
\newcommand\oprocend{\relax\ifmmode\else\unskip\hfill%
\fi\oprocendsymbol}

\newcommand{\Ply}{P}      
\newcommand{\Qly}{Q}      

\newcommand{\map}[3]{#1: #2 \rightarrow #3}
\newcommand{\setdef}[2]{\{#1 \; : \; #2\}}

\newcommand{\sbs}[2]{{#1}_{\textup{#2}}}
\newcommand{\sps}[2]{{#1}^{\textup{#2}}}

\newcommand{\norm}[1]{\Vert #1 \Vert}
\newcommand{\distA}[1]{\vert #1 \vert_{\mc U_k^*}}
\newcommand{\dist}[2]{\operatorname{dist}(#1,#2)}

\usepackage{hyperref}
\usepackage{graphicx,color}
\graphicspath{{./IMG/}{images/}{./IMG/SAGEMATH_PLOTS/}}
\usepackage{amsmath}
\usepackage{amssymb}
\usepackage{mathrsfs}
\usepackage{subfigure}
\usepackage{url}
\usepackage{booktabs}
\usepackage{array}
\usepackage[table]{xcolor}
\usepackage{mathtools} 		
\usepackage{epstopdf}
\usepackage{cite}
\usepackage[vlined,ruled]{algorithm2e}
\usepackage{upgreek}
\usepackage{dsfont}

\usepackage{epstopdf}
\title{\LARGE \textbf{Data-Driven Synthesis of Optimization-Based Controllers for \\ Regulation of Unknown Linear Systems}}

\author{Gianluca Bianchin \quad Miguel Vaquero \quad Jorge Cort\'{e}s \quad Emiliano Dall'Anese\thanks{
This work was supported by the National Science Foundation Awards CMMI 2044946 and 2044900, and by  NREL  through  the subcontract UGA-0-41026-148. G. Bianchin and E. Dall'Anese are with the Department of Electrical, Computer, and Energy Engineering, University of Colorado Boulder. M. Vaquero is with the School of Human Sciences and Technology, IE University. J. Cort\'es is with the Department of Mechanical and Aerospace Engineering, University of California San Diego. }}

\begin{document}
\maketitle

\begin{abstract}
This paper proposes a data-driven framework to solve time-varying  
optimization problems associated with unknown linear dynamical  
systems.
Making online control decisions to regulate a dynamical system to the 
solution of an optimization problem is a central goal in 
many modern engineering applications. Yet, the available methods 
critically rely on a precise knowledge of the system dynamics, thus 
mandating a preliminary system identification phase before a 
controller can be designed. In this work, we leverage results 
from behavioral theory to show that the steady-state transfer 
function of a linear system can be computed from data samples without any knowledge or estimation 
of the system model. We then use this data-driven representation to design a 
controller, inspired by a gradient-descent optimization method, that 
regulates the system to the solution of a convex optimization problem, 
without requiring any knowledge of the time-varying disturbances 
affecting the model equation. 
Results are tailored to cost functions satisfy the Polyak-\L ojasiewicz inequality.
\end{abstract}

\section{Introduction}
Online optimization problems have attracted significant attention 
in various disciplines, including machine learning~\cite{SS:11},  
control systems~\cite{MC-ED-AB:20,LL-ZN-EM-JS:18}, and transportation 
management~\cite{GB-JC-JP-ED:21-tcns}.
When used as a controller for dynamical systems, an online 
optimization method seeks to make control decisions at every time 
instant to minimize a loss function that is time-varying and possibly 
uncertain as described by the system 
dynamics.
The vast majority of works on online optimization for dynamical 
systems make a strict assumption on the knowledge of the underlying 
system dynamics (see e.g. \cite{MC-ED-AB:20,AH-SB-GH-FD:20,menta2018stability,GB-JIP-ED:20-automatica}).
However, besides their theoretical value, maintaining and refining 
full system models is often undesirable because: (i) perfect 
knowledge of the dynamics is rarely available in practice since
it requires explicit system identification and periodic model
re-updates, and (ii) identifying a full model of the system is often 
unnecessary since most optimization-inspired controllers rely on 
simpler representations of the dynamical system. To the best of our 
knowledge, efficient and numerically reliable online optimization 
methods that bypass the model identification phase are
still  lacking.

In this work, we take a novel approach to design online optimization
controllers that relies on Willems' fundamental lemma
\cite{JCW-PR-IM-BDM:05} to construct a data-driven representations of 
the dynamical system to control, which is then used for algorithm 
synthesis.
Precisely, we assume the availability of noise-free historical 
data, i.e., finite-length trajectories produced by the open-loop 
dynamics, and we show that the steady-state transfer function of a 
linear system can be computed from non steady-state input-output data.
The noise-free assumption corresponds to scenarios where accurate 
sensors or signal-processing techniques can be utilized offline to 
recover noise-free sample trajectories (a similar approach was taken 
in~\cite{coulson2019data,LX-MT-BG-GF:21}). 
Then, we build upon such data-driven representation to 
propose a controller inspired by online optimization methods that 
regulates the dynamical  system to an equilibrium point that 
minimizes a prescribed loss function despite unknown and time-varying 
noise terms affecting the model equation.
Interestingly, our results suggest that a suitable choice of the 
controller stepsize is sufficient to guarantee asymptotic convergence 
to the desired optimizers, up to an asymptotic error that is bounded 
by the time-variability of the exogenous noise terms.

\noindent \textbf{Related Work.}
The results presented here are tied to the fields of data-driven 
control and online optimization. 
The success of data-driven control methods mainly originates from 
the possibility of synthesizing controllers without first 
identifying a full system model.
Among these methods, the behavioral framework has recently regained 
considerable attention 
\cite{JCW-PR-IM-BDM:05,CDP-PT:19}. 
Recent extensions include distributed formulations \cite{AA-JC:20}, 
combinations with model predictive control
\cite{coulson2019data,JB-JK-MAM-FA:20}, 
trajectory tracking \cite{LX-MT-BG-GF:21}, 
and nonlinear systems \cite{JB-FA:20,MG-CDP-PT:20}.
In this work, we leverage the behavioral framework to build a 
data-driven representation of a dynamical system that is the used 
for optimization purposes.
Especially relevant to this work are the recent results
\cite{MY-AI-RS:20,HJVW-MKC-MM:20,AB-CD-PT:21,LX-MT-BG-GF:21}
that focus on the presence of noise in the data.

Online optimization approaches aim to optimize loss functions that 
depends on the state of an underlying and uncertain dynamical system.
Linear time-invariant systems are considered in 
e.g.,~\cite{MC-ED-AB:20,menta2018stability,GB-JC-JP-ED:21-tcns,LL-ZN-EM-JS:18},
stable nonlinear systems in~\cite{AH-SB-GH-FD:20,DL-DF:SM:21}, and 
switched systems in \cite{GB-JIP-ED:20-automatica}. In contrast with
the above line of work, which considers continuous-time dynamics, 
the focus of this paper is on systems and controllers that operate 
at discrete time.
Although data-driven online optimization methods are studied in the 
recent work~\cite{MN-MM:21}, results are however limited to the 
absence of noise and regret~analysis.






\noindent \textbf{Contributions.}
This work features two main contributions. 
First, we show that the steady-state transfer function of a linear 
time-invariant dynamical system can be obtained from historical (non 
steady-state) input-output trajectories generated by the open-loop 
dynamics, without any knowledge or estimation of the system 
parameters.
Interestingly, our results also suggest that the steady-state transfer
function can be computed exactly from input-output data even when 
the trajectories are affected by constant noise terms.
Our work offers contributions with respect to the vast majority of the 
available literature on data-driven control, which considers 
disturbances affecting only the output equation
(see e.g. \cite{MY-AI-RS:20,AB-CD-PT:21}), by accounting for the 
presence of disturbance terms affecting the model equation.
Second, we propose a controller inspired from online optimization 
methods, which steers the dynamical system to one solution of a 
time-varying convex optimization problem. We prove convergence of the 
system with controller to the desired optimal points; precisely, we 
show input-to-state stability (ISS) of the controlled dynamical system 
with respect to exogenous disturbances affecting the dynamics.
Our results build upon the theory 
of ISS Lyapunov functions for discrete-time dynamical systems 
\cite{ZPJ-YW:01}, properly modified to guarantee stability with 
respect to compact sets of optimizers \cite{ES-YW:95}.


\section{Preliminaries}
\label{sec:2}

We first outline the notation and recall few basic concepts.

\noindent
\textbf{Notation.}
Given a symmetric matrix 
$M \in \real^{n \times n}$, $\underline \lambda(M)$ and $\bar 
\lambda(M)$ denote the smallest and largest eigenvalue of $M$, 
respectively; $M \succ 0$ indicates that $M$ is positive definite.
For vectors $u \in \real^n$ and $w \in \real^m$, 
$(x,u) \in \real^{n+m}$ denotes their concatenation.
We denote by $\norm{u}$ the Euclidean norm of $u$; $u^\top$ denotes transposition; given nonempty 
compact sets $\mc A, \mc B \subset \real^n$, 
$\vert u \vert_{\mc A} = \inf_{z \in \mc A} \norm{z-u}$ denotes the
point-to-set distance, while 
$\dist{\mc A}{\mc B} := \max \{ 
\sup_{x \in \mc A} \inf_{y \in \mc B} \norm{x-y}, 
\sup_{y \in \mc B} \inf_{x \in \mc A} \norm{x-y} \} $ denotes the 
Hausdorff distance.

A continuous function 
$\beta:\mathbb{R}_{\geq0}\times\mathbb{R}_{\geq0}\to\mathbb{R}_{\geq0}$ is of class $\mathcal{K}\mathcal{L}$ if it is strictly increasing in 
its first argument, decreasing in its second argument, 
$\lim_{r\to0^+}\beta(r,s)=0$ for each $s\in\mathbb{R}_{\geq0}$, and  
$\lim_{s\to\infty}\beta(r,s)=0$ for each $r\in\mathbb{R}_{\geq0}$. 
A continuous function $\map{\gamma}{\realpos}{\realpos}$ is of class 
$\mc{K}$ if it is strictly increasing and $\gamma(0)=0$, and   it is
of class $\mc K_\infty$ if it is of class $\mc K$ and, in addition,
$\lim_{r\rightarrow \infty} \gamma(r)=\infty$.

\noindent
\textbf{Persistency of Excitation.}
We recall some useful facts on behavioral system theory 
from~\cite{JCW-PR-IM-BDM:05}. 
For a signal  $k \mapsto z_k \in \real^\sigma$, 
$k \in \integpos$, we denote by 
$z_{[k,k+T]}$, $k \in \integ$, $T \in \integpos$, 
the vectorization of $z$ restricted to the interval $[k, k + T ]$, 
namely,
\begin{align*}
 z_{[k,k+T]} = (z_k, \dots, z_{k+T}).
\end{align*}
Given $z_{[0,T-1]}$, $t \le T$, and $q\leq T-t+1$, we let $Z_{t,q}$ 
denote the Hankel matrix of length $t$ associated with $z_{[0,T-1]}$:
\begin{align*}
Z_{t,q} = \begin{bmatrix}
z_0 & z_1 & \hdots & z_{q-1}\\
z_1 & z_2 & \hdots & z_q\\
\vdots & \vdots & \ddots & \vdots\\
z_{t-1} & z_t & \hdots & z_{q+t-2}
\end{bmatrix}
\in \real^{\sigma t \times q}.
\end{align*}
The signal $z_{[0,T-1]}$ is persistently 
exciting of order $t$ if $Z_{t,q}$ has full row rank; that is, 
$\rank(Z_{t,q}) = \sigma t$. Notice that persistency of excitation 
implicitly requires $q \geq \sigma t$ and, consequently, 
$T \geq (\sigma+1) t -1$.

The linear dynamical system
\begin{align}
\label{eq:auxDynamicalSystem}
x_{k+1} &= A x_k + B u_k, & y_k &= C x_k + D u_k, 
\end{align}
$x \in \real^{n}$, 
is controllable if 
$\mc C := [B, AB, A^2B, \dots, A^{n-1}B]$ satisfies 
$\rank(\mc C) = n$. We recall the following two properties of 
\eqref{eq:auxDynamicalSystem} when its input is persistently exciting. 


\begin{lemma}{\textit{\textbf{(Fundamental Lemma)}}
\cite[Corollary 2]{JCW-PR-IM-BDM:05}}
\label{lem:fundLemmarankHankelMatrix}
Assume \eqref{eq:auxDynamicalSystem} is controllable, let 
$(u_{[0,T-1]}, x_{[0,T-1]})$, $T \in \integ_{>0}$, be an 
input-state trajectory of \eqref{eq:plantModel}. 
If $u_{[0,T-1]}$ is persistently exciting of order $n+L$, then:
\begin{align*}
\rank \begin{bmatrix} U_{L,q}\\  X_{1,q} \end{bmatrix} 
= L m + n,
\end{align*}
where $U_{L,q}$ and $X_{1,q}$ denote the Hankel matrices associated with $u_{[0,T-1]}$ and $x_{[0,T-1]}$, respectively. \hfill $\Box$
\end{lemma}
\smallskip

\begin{lemma}{\bf \textit{(Data Characterizes Full Behavior \cite[Theorem 1]{JCW-PR-IM-BDM:05}})}
\label{lem:fundLemmaExistenceg}
Assume \eqref{eq:auxDynamicalSystem} is controllable and observable, 
let $(u_{[0,T-1]}, y_{[0,T-1]})$, $T \in \integ_{>0}$, be an 
input-output trajectory generated by \eqref{eq:plantModel}. 
If $u_{[0,T-1]}$ is persistently exciting of order $n+L$, then
any pair of $L$-long signals 
$(\tilde u_{[0,L-1]}, \tilde y_{[0,L-1]})$ is an input-output 
trajectory of \eqref{eq:plantModel} if and only if there  exists
$\alpha \in \real^q$ such that:
\begin{align*}
\begin{bmatrix}\tilde u_{[0,L-1]}\\ \tilde y_{[0,L-1]}\end{bmatrix}
= 
\begin{bmatrix} U_{L,q} \\ Y_{L,q}\end{bmatrix}
\alpha,
\end{align*}
where $U_{L,q}$ and $Y_{L,q}$ denote the Hankel matrices associated 
with $u_{[0,T-1]}$ and $y_{[0,T-1]}$, respectively. \hfill $\Box$
\end{lemma}
\smallskip

In words, persistently exciting signals generate output trajectories
that can be used to express any other trajectory.

\section{Problem Formulation}
\label{sec:3}
We consider discrete-time linear time-invariant dynamical systems 
subject to state noise, described by:
\begin{align}
\label{eq:plantModel}
x_{k+1} &=  A x_{k} + B u_{k} + E w_{k}, & 
y_k &= C x_{k},
\end{align}
where $k \in \integpos$ is the time index, $x_k \in \real^n$ is the 
state,  $u_k \in \real^m$ describes the control
decision, $w_k \in \real^r$ denotes an unknown exogenous input 
or disturbance, which is assumed to be bounded at all times, 
$y_k \in \real^n$ is the measurable output, 
and $A, B, C, E$ are matrices of suitable dimensions.
We assume that any equilibrium point of~\eqref{eq:plantModel} is
asymptotically stable, as formalized next.

\begin{assumption}\label{ass:stabilityPlant}{\bf\textit{(Stability of Plant)}}
The matrix $A$ is Schur stable, i.e., for any $Q\succ 0$ there exists
$P \succ 0$ such that $A^\tsp P A - P = -Q$.
Moreover, \eqref{eq:plantModel} is controllable and the columns of $C$
are linearly independent.
\QEDB
\end{assumption}

%
\begin{remark}\label{re:C-full-column-rank}{\bf \textit{(Linear Independence Columns of $C$)}}
Linear-independence of the columns of $C$ 
implies that the state of \eqref{eq:plantModel} can be fully 
determined at every time given the system output. 
This assumption can be relaxed, as shown in \cite{GB-MV-JC-ED:21-tac}.
\QEDB
\end{remark}


Under Assumption \ref{ass:stabilityPlant}, for any fixed 
$u \in \real^m$ and $w \in \real^r$, the system 
\eqref{eq:plantModel} admits a well-posed steady-state 
input-output relationship, given by:
\begin{align}
\label{eq:steady-state}
y = \underbrace{C(I-A)^\inv B}_{:= G} u
+ \underbrace{C(I-A)^\inv E}_{:=H } w.
\end{align}

We focus on cases where the matrices $(A,B,C,E)$ are unknown, and we 
consider the problem of regulating~\eqref{eq:plantModel} to an 
equilibrium point that is specified as the solution of the following 
optimization problem:
\begin{align}
\label{opt:objectiveproblem}
u_k^* \in \arg \min_{\bar u} \;\;\; 
& \phi(\bar u) + \psi(G \bar u + H w_k),
\end{align}
where $\map{\phi}{\real^m}{\real}$ and
$\map{\psi}{\real^p}{\real}$ are given cost functions describing 
losses associated withe the system input and output, respectively. 
We note that, at every time $k$, the value of the objective function
in \eqref{opt:objectiveproblem} is unknown and time-varying, since it 
depends on the unknown matrices $(A,B,C,E)$ and on the unknown and 
time-varying disturbance $w_k$.
The optimization problem \eqref{opt:objectiveproblem} formalizes an 
optimal regulation problem, where the goal is to regulate 
\eqref{eq:plantModel} to an optimal equilibrium point, as described by
the cost in \eqref{opt:objectiveproblem}.

We make the following assumption on the costs 
in~\eqref{opt:objectiveproblem}.


\begin{assumption}\label{ass:lipschitzAndConvexity}{\bf\textit{(Lipschitz Smoothness and PL)}}
The functions $u \mapsto \phi(u)$ and $y \mapsto \psi(y)$ are 
differentiable and have Lipschitz-continuous gradients with constants
$\ell_\phi$, $\ell_{\psi}$, respectively.
Moreover, $f(u):=\phi(u) + \psi(Gu+Hw)$ is radially-unbounded, has a 
nonempty set of minimizers, and satisfies the Polyak-\L ojasiewicz 
(PL) inequality, i.e., there exists $\mu>0$ such that 
$\frac{1}{2} \norm{\nabla f(u)}^2 \geq \mu ( f(u) - f(u^*))$, 
for all $u\in \real^m$ and all minimizers~$u^*$.
\QEDB
\end{assumption}


Lipschitz-continuity assumptions are commonly used for the study of 
first-order optimization methods (see e.g.~\cite{SB-LV:04}).
The PL inequality allows us to guarantee that every critical point of 
\eqref{opt:objectiveproblem} is a global minimizer.
Also, we note that the PL condition is a weaker assumption than strong 
convexity (in particular, it implies invexity), and has been 
widely adopted to study convergence of optimization  
algorithms~\cite{SB-LV:04}.

We seek to synthesize a controller that does not require any prior 
knowledge of the matrices $(A,B,C,E)$ as well as of the disturbance 
$w_k$, with the following structure:
\begin{align}
\label{eq:controllerFc}
u_{k+1} = \sbs{F}{c}(u_k, y_k),
\end{align}
that guarantees that \eqref{eq:plantModel} tracks the optimizers of 
the optimization problem~\eqref{opt:objectiveproblem} (see 
Fig.~\ref{fig:interconnectionBlocks} for an illustration). 
Two important observations are in order. 
First, because the exogenous input $w_k$ is  time-varying, 
the optimizers of \eqref{opt:objectiveproblem} are also time-varying.
Formally, by denoting by 
$\mc{U}_k^*:= \setdef{u_k^*}{ 0= \nabla \phi(u_k^*) 
+ \nabla \psi(Gu_k^* + Hw_k)}$ the set of optimizers of 
\eqref{opt:objectiveproblem} at time $k$, in general we have that 
$\mc U_{k+1}^* \neq \mc U_{k}^*$.
Second, because we assume no prior knowledge on $w_k$, any controller 
of the form \eqref{eq:controllerFc} 
can track the solutions of \eqref{opt:objectiveproblem} up to an error 
that depends on the time-variability of $w_{k+1}-w_k$.
For these reasons, we aim at guaranteeing that the output of the 
system \eqref{eq:plantModel} satisfies a tracking bound of the form:
\begin{align}
& \hspace{-.2cm} \distA{\xi_k -\xi^*_k} \leq 
\beta (|\xi_{k_0} - \xi^*_{k_0}|_{\mathcal{U}^*_{k_0}}, k-k_0) \nonumber \\
& \hspace{.2cm}
+ \gamma_u (\sup_{ t \geq k_0} \dist{\mc U_{t+1}^*}{\mc U_{t}^*})
+ \gamma_w (\sup_{ t \geq k_0} \norm{w_{t+1} - w_t}) \label{eq:KLbound}
\end{align}
for all $0 \leq k_0 \leq k$, where
$\xi_k := (x_k,u_k)$ denotes the joint system-controller state,  
$\xi^*_k := ((I-A)^\inv(B u_k^*+Ew_k),u^*_k)$ denotes 
the optimizer of \eqref{opt:objectiveproblem} at time $k$, $\beta$
is a class-$\mc{KL}$ function, $\gamma_u$, $\gamma_w$
are class-$\mc K$ functions, and 
$\dist{\mc U_{t+1}^*}{\mc U_{t}^*}$ denotes the Hausdorff 
distance~\cite{RR-RW:09} between the compact sets $\mc U_{t+1}^*$ and 
$\mc U_{t}^*$. 
Observe that, because the cost function is radially-unbounded (see 
Assumption~\ref{ass:lipschitzAndConvexity}), the set of optimizers of 
\eqref{opt:objectiveproblem} is compact and thus the Hausdorff 
distance $\dist{\mc U_{t+1}^*}{\mc U_{t}^*}$ is finite for all 
$k$~\cite{RR-RW:09}.

\begin{remark}{\bf\textit{(Time-Varying Optimizers)}}
We notice that although $\dist{\mc U_{t+1}^*}{\mc U_{t}^*}$ 
may be reconducted to the time-variability of $w_{k}$,
in general \eqref{eq:KLbound}  reflects more accurately the role of 
individual error terms, since the precise relationship between 
$w_{k+1}-w_k$ and $\dist{\mc U_{t+1}^*}{\mc U_{t}^*}$ is unknown 
for general cases.
Furthermore, in cases where the loss $\mc \phi(u)$ and 
$\psi(y)$ are time-varying, the term 
$\dist{\mc U_{t+1}^*}{\mc U_{t}^*}$ can be used to account for the
time-variability of the optimizers.~
\QEDB
\end{remark}

\begin{figure}[t]
\centering
\includegraphics[width=.7\columnwidth]{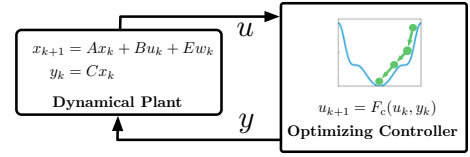}
\caption{Online gradient-flow optimizer used as an output feedback
controller for unknown LTI systems subject to time-varying 
disturbances.}
\label{fig:interconnectionBlocks}
\vspace{-.4cm}
\end{figure}

\section{Data-Driven Method for Online Optimization}
\label{sec:4}

To track the optimizers of \eqref{opt:objectiveproblem}, we propose
the following controller inspired from an online gradient
method:
\begin{align}
\label{eq:gradientFlowController}
x_{k+1} &=  A x_k + B u_k + E w_k, ~~ y_k = C x_k, \nonumber\\
u_{k+1} &= u_k -\eta (\nabla \phi(u_k) + G^\tsp \nabla \psi(y_k)),
\end{align}
where $\eta \in \real_{>0}$ is a tunable controller parameter
(see Fig.~\ref{fig:interconnectionBlocks}).
We note that the controller in~\eqref{eq:gradientFlowController} does 
not rely on any knowledge of the system matrices $(A,B,C,E)$, 
instead, it requires an exact expression for the map $G$. 
In order to implement \eqref{eq:gradientFlowController}, we propose a 
two-phase control method, where data samples are used to determine 
$G$, and then the feedback controller 
\eqref{eq:gradientFlowController} is used to track the optimizers 
of~\eqref{opt:objectiveproblem}.  
Fig.~\ref{fig:historical_control_phase} illustrates the two phases.
Similar approaches using data recorded offline were proposed 
in~\cite{LX-MT-BG-GF:21} for output tracking as 
well as in~\cite{coulson2019data} for model predictive~control.

\subsection{Data-Driven Characterization of the Transfer Function}
The following result shows that, when $w_k=0$ at all times, the 
steady-state transfer function $G$ can be computed from a
(non-steady-state) sample trajectory of the system.


\begin{theorem}{\bf \textit{(Data-Driven Characterization of 
Transfer Function in the Absence of Noise)}}
\label{thm:data-drivenTransferFunction}
Let Assumption \ref{ass:stabilityPlant} hold, let 
$u_{[0,T-1]}$ be persistently exciting of order $n+1$
and assume $W_{1,T}=0$.
Then, there exists $M \in \real^{q \times m}$ such that:
\begin{align}
\label{eq:definionMatrixM}
\sps{Y}{diff}_{1,T}M=0, \qquad
U_{1,T}M=I, 
\end{align}
where $\sps{Y}{diff}_{1,T} = [y_1- y_0, ~ y_2-y_1,~ 
\dots ~ y_T-y_{T-1}].$
Moreover, for any $M$ that satisfies \eqref{eq:definionMatrixM}, 
the steady-state transfer function of \eqref{eq:plantModel}
equals $G=  Y_{1,T} M$.  \hfill $\Box$
\end{theorem}

The proof is presented in \cite{GB-MV-JC-ED:21-tac}.
Theorem~\ref{thm:data-drivenTransferFunction} asserts that $G$ can be 
computed from sample data originated from~\eqref{eq:plantModel}, by 
solving the set of linear equations~\eqref{eq:definionMatrixM}.
Two important observations are in order. 
First, the matrix $M$ that satisfies~\eqref{eq:definionMatrixM} 
depends on the sample data $u_{[0,T-1]}$. Second, given a sample 
sequence $u_{[0,T-1]}$, in general, 
there exists an infinite number of choices of $M$ that satisfy 
\eqref{eq:definionMatrixM}.
Despite  $M$ not being unique, 
Theorem~\ref{thm:data-drivenTransferFunction} shows that $Y_{1,q} M$ 
is unique and independent of the choice of $u_{[0,T-1]}$ used to 
generate the data.
It is worth noting that Theorem~\ref{thm:data-drivenTransferFunction} 
requires one to collect $T$ samples of the (persistently exciting) 
input sequence $u_{[0,T-1]}$, and $T+1$ samples of the associated 
output sequence $y_{[0,T]}$.

\begin{figure}[t]
\centering
\includegraphics[width=.9\columnwidth]{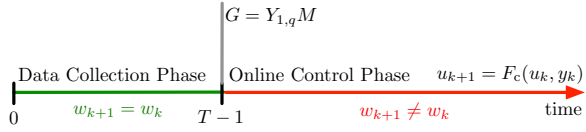}
\caption{Two-phase control method where noise-free historical data is 
used to compute the map $G$, and then a dynamical controller is used 
to track the optimizers of \eqref{opt:objectiveproblem} despite 
unknown and time-varying disturbances.}
\label{fig:historical_control_phase}
\end{figure}

Theorem \ref{thm:data-drivenTransferFunction} assumes the availability
of a finite-length trajectory produced by the open-loop system 
\eqref{eq:plantModel} in the absence of exogenous disturbance $w_k$. 
When $w_k$ is non-zero but constant at all times, Theorem 
\ref{thm:data-drivenTransferFunction} can still be used to determine 
$G$, as described next. For all $k \in \integpos$, define
\begin{align}
\label{eq:systemDifferences}
d_k := x_{k+1}-x_k, 
r_k := y_{k+1}-y_k,
v_k := u_{k+1}-u_k.
\end{align}
By substituting into \eqref{eq:plantModel}, the new variables satisfy:
\begin{align}
\label{eq:update-d}
d_{k+1} &= A d_k + B v_k,& r_k &= C d_k.
\end{align}
By leveraging \eqref{eq:update-d}, 
Theorem~\ref{thm:data-drivenTransferFunction} can be used to determine
$G$ when the sample  data is affected by constant disturbance.


\begin{corollary}{\bf \textit{(Data-Driven Characterization of 
Transfer Function With Constant Noise)}}
\label{cor:constantNoiseG}
Let Assumption \ref{ass:stabilityPlant} hold, assume
$u_{[0,T]}$ is persistently exciting of order $n+1$. If 
$w_k=w \in \real^r$ is fixed for all $k \in \integpos$, 
then the steady-state transfer function of 
\eqref{eq:plantModel} equals $G=  R_{1,q} M$, where 
\begin{align}
\label{eq:definionMatrixM-noisy}
\sps{R}{diff}_{1,q}M=0, \qquad
V_{1,q}M =I, 
\end{align}
$\sps{R}{diff}_{1,q} = [r_1- r_0, ~ r_2-r_1,~ 
\dots ~ r_q-r_{q-1}]$, and $q = T+1$. \hfill $\Box$
\end{corollary}


Corollary \ref{cor:constantNoiseG} provides a direct way to compute 
the transfer function $G$ when the sample data is affected by 
constant noise. 
Notice that, the Hankel matrices $R_{1,q}, V_{1,q}$, and
$\sps{R}{diff}_{1,q}$ can be computed directly from an input-output 
trajectory of \eqref{eq:plantModel} by preprocessing the data samples 
as described by  \eqref{eq:systemDifferences}.

\subsection{Convergence to First-Order Optimizers}
We now turn our attention to the online control phase, where the 
map $G$ computed according to 
Theorem~\ref{thm:data-drivenTransferFunction} is used in the 
gradient-based controller in~\eqref{eq:gradientFlowController}. 
The following result guarantees global convergence to the set of 
optimizers of \eqref{opt:objectiveproblem} under a suitable choice of 
the stepsize~$\eta$. 


\begin{theorem}\label{thm:convergence}{\bf \textit{(Tracking of Optimal Solutions)}}
Let Assumptions~\ref{ass:lipschitzAndConvexity}-\ref{ass:stabilityPlant} hold, 
let $\ell = \ell_\phi + \norm{G}^2 \ell_\psi$, and assume that $Q$ satisfies
$\underline \lambda (\Qly) > \frac{a}{\epsilon ( 1 -\epsilon)}$, where $a = \frac{1}{2} \ell_\psi^2 \norm{C}^2\norm{G}^2$ and $\epsilon \in (0,1)$ is a fixed parameter. Moreover, let
\begin{align*}
\eta^* &:= \frac{1-\epsilon}{\ell/2 + b}, &
b :=  \frac{2 \norm{A^\tsp \Ply \bar G}^2}{\epsilon 
\underline \lambda(\Qly)} + \norm{\bar G^\tsp \Ply \bar G},
\end{align*}
where $\bar G :=(I-A)^\inv B$. Then, for every $\eta < \eta^*$ 
there exists a class-$\mc{KL}$ function $\beta$ and class-$\mc K$ 
functions $\gamma_u$, $\gamma_w$ such that all solutions of \eqref{eq:gradientFlowController} satisfy 
\eqref{eq:KLbound}. \QEDB
\end{theorem}


The proof of Theorem \ref{thm:convergence} relies on the following 
result, which can be interpreted as an extension of the 
characterization of input-to-state stability for equilibrium points 
studied in \cite{ZPJ-YW:01} to the case of compact forward invariant 
sets \cite{ES-YW:95}.


\begin{lemma}{\bf\textit{(ISS with Respect to Compact Sets)}}
\label{lem:ISS}
Consider the system $z_{k+1} = f(z_k,v_k)$ where 
$\map{f}{ \real^n \times \real^m}{\real^n}$ is locally 
Lipschitz and $v_k$ is a bounded input sequence.
Let $\mc A \subset \real^n$ be a nonempty and compact  set that is 
forward invariant for the unforced system. Let $\map{V}{\real^n}{\realpos}$ be a continuous function such that
\begin{subequations}
\begin{align}
\label{eq:ISS-a}
\alpha_1(\vert z \vert_{\mc A}) &\leq V(z) 
\leq \alpha_2 (\vert z \vert_{\mc A}),\\
\label{eq:ISS-b}
V(f(z,v))-V(z) &\leq - \alpha_3 (\vert z \vert_{\mc A}) + \sigma(\vert v \vert_{\mc A}),
\end{align}
\end{subequations}
hold for all $z \in \real^n$, and 
$v \in \real^m$, where $\alpha_1, \alpha_2, \alpha_3$ are 
class $\mc{K}_\infty$ functions and $\sigma$ is of class
$\mc{K}$. Then, there exists a class $\mc{KL}$ 
function $\beta$ and a class $\mc K$ function $\gamma$ such that the 
system solutions satisfy
\begin{align*}
\vert z_k \vert_{\mc A} \leq \beta(\vert z_{k_0} \vert_{\mc A},k-k_0) 
+ \gamma (\sup_{t \geq k_0} \norm{v_t}),
\end{align*}
for all $0 \leq k_0 \leq k$, and for any $z_{k_0} \in \real^n$. \hfill $\Box$
\end{lemma}
\begin{proof}
The claim follows by iterating the proof of 
\cite[Lemma 3.5]{ZPJ-YW:01}. Precisely, we note that, when $\mc A$ is 
nonempty and compact, the quantity $\vert z\vert_{\mc A}$ is 
well-defined and bounded for any $z \in \real^n$, and thus all 
Euclidean norms in \cite[Lemma 3.5]{ZPJ-YW:01} can be replaced by the 
point-to-set distance $\vert \cdot \vert_{\mc A}$.
\end{proof}

\begin{pfof}{Theorem \ref{thm:convergence}}
We begin by performing a change of variables for 
\eqref{eq:gradientFlowController}. 
Let $\tilde x_k := x_k - \mc M(u_k, w_k)$, where $\mc M(u,w) = \bar G u - \bar H w$,  $\bar G :=(I-A)^\inv B$, $\bar H :=(I-A)^\inv E $.
In the new variables:
\begin{align}
\label{eq:changeVariables}
\tilde x_{k+1} &=  A \tilde x_k- \mc M(u_{k+1},w_{k+1})  
+ \mc M(u_k,w_k)\\
&= A \tilde x_{k} - \bar G (u_{k+1} - u_{k})
- \bar H (w_{k+1} - w_{k}), \nonumber\\
u_{k+1} &= u_{k} - \eta (\nabla \phi(u_{k}) 
+ G^\tsp \nabla \psi(C \tilde x_k + Gu_k + Hw_k)). \nonumber
\end{align}
Next, let $f(u) := \phi(u) + \psi(Gu+Hw_k)$, and 
denote by $u^*$ any minimizer of $f(u)$.
We will show that the following Lyapunov function satisfies 
the assumptions of Lemma \ref{lem:ISS}:
\begin{align}
\label{eq:Udefinition}
U(x, u) :=  V(u) +  W(\tilde x),
\end{align}
where $V(u) = \frac{1}{\eta} (f(u)- f(u^*))$, 
$W(\tilde  x) = \tilde  x^\tsp \Ply \tilde  x$.

First, \eqref{eq:ISS-a} follows by application 
of~\cite[Lemma 4.3]{HKK:02} by noting that $U(x,u)$ is continuous, 
positive definite, and radially unbounded.
Next, we show \eqref{eq:ISS-b}.
By letting
$\sbs{F}{c}(\tilde  x, u): =  
-\nabla \phi(u) - G^\tsp \nabla \psi(C \tilde x + Gu + Hw_k)$,
and by noting that $\nabla f(u) = -\sbs{F}{c}(0,u)$, $V(\cdot)$ 
satisfies:
\begin{align}
\label{aux:Vdot-1}
& V(u_{k+1}) - V(u_k)  = 
\frac{f(u_{k+1}) - f(u_k)}{\eta} - \frac{f(u_{k+1}^*) + f(u_k^*)}{\eta}  
\nonumber \\
&\quad \leq \langle \nabla f(u_k), \frac{ u_{k+1} - u_k}{\eta} \rangle  + \frac{\ell}{2 \eta} \distA{u_{k+1}-u_k}^2 \nonumber\\
&\quad\quad\quad\quad + 
\underbrace{\langle \nabla f(u_k^*), \frac{ u_{k+1}^* - u_k^*}{\eta}}_{=0} \rangle  
+ \frac{\ell}{2 \eta} \norm{u_{k+1}^*-u_k^*}^2 \nonumber\\
&\quad \leq -  
(1 - \frac{\ell \eta}{2})\distA{\sbs{F}{c}(\tilde x_k,u_k)}^2 
+ \frac{\ell}{2 \eta} \norm{u_{k+1}^*-u_k^*}^2 \nonumber\\
&\quad\quad\quad\quad 
+ \ell_\psi \norm{C}\norm{G}\distA{\tilde x} \distA{\sbs{F}{c}(\tilde x_k,u_k)},
\end{align}
where the first inequality follows from 
$f(u) - f(v) \leq \nabla f(v)^\tsp (u-v) 
+ \frac{\ell}{2} \norm{u-v}^2$, which holds $\forall u,v\in \real^m$ 
under Assumption \ref{ass:lipschitzAndConvexity}, and the last 
inequality follows by noting that 
$\distA{u_{k+1}-u_k}^2 = \eta^2 \distA{\sbs{F}{c}(\tilde x_k,u_k)}^2$
and by using:
\begin{align*}
& \langle \nabla f(u_k), \frac{ u_{k+1} - u_k}{\eta} \rangle 
= - \langle \sbs{F}{c}(0,u_k) ,\sbs{F}{c}(\tilde x_k,u_k) \rangle \\
& \quad\quad = 
- \langle \sbs{F}{c}(0,u_k) + \sbs{F}{c}(\tilde x_k,u_k) - \sbs{F}{c}(\tilde x_k,u_k),  \sbs{F}{c}(\tilde x_k,u_k) \rangle\\
& \quad\quad \leq 
- \distA{\sbs{F}{c}(\tilde x_k,u_k)}^2 + \ell_\psi \norm{C} \norm{G} 
\norm{\tilde x_k} \distA{\sbs{F}{c}(\tilde x_k,u_k)},
\end{align*}
where we used Assum.~\ref{ass:lipschitzAndConvexity}.
By completing the squares~in~\eqref{aux:Vdot-1}:
\begin{align}
\label{aux:Vdot-2}
V(u_{k+1}) &- V(u_k)  \leq
- (1 - \frac{\epsilon}{2}- \frac{\ell \eta}{2})
\distA{\sbs{F}{c}(\tilde x_k,u)}^2  
\\
& + \underbrace{
\frac{\ell_\psi^2 \norm{C}^2\norm{G}^2}{2 \epsilon} }_{:=a_1}
\distA{\tilde x_k}^2
+ \underbrace{\frac{\ell}{2 \eta}}_{:=a_2} \norm{u_{k+1}^*-u_k^*}^2.  \nonumber
\end{align}
\vspace{-.2cm}

By denoting in compact form $\Delta w_k := w_{k+1} - w_k$
and by recalling that 
$u_{k+1}-u_k = \eta \sbs{F}{c}(\tilde x_k, u_k)$, $W(\cdot)$ 
satisfies:
\begin{align*}
&W(\tilde x_{k+1}) - W(\tilde x_k ) = 
 \tilde x^\tsp_k (A^\tsp \Ply A - \Ply)  \tilde x_k 
\nonumber \\
& \quad 
+ \eta^2 \sbs{F}{c}^\tsp(\tilde x_k, u_k)  \bar G^\tsp \Ply \bar G 
\sbs{F}{c}(\tilde x_k, u_k)  
+ \Delta w_k^\tsp \bar H^\tsp \Ply \bar H \Delta w_k 
\nonumber\\
&\quad  
- 2 \eta \tilde x_k^\tsp A^\tsp \Ply \bar G  
\sbs{F}{c}(\tilde x_k, u_k) 
- 2 \tilde x_k^\tsp A^\tsp \Ply \bar H \Delta w_k 
\nonumber\\
& \quad  
+ 2 \eta \sbs{F}{c}^\tsp(\tilde x_k, u_k) \bar G^\tsp \Ply \bar H \Delta w_k 
\nonumber \\
&\leq
- \underline \lambda(\Qly) \distA{x_k}^2 
+ 2 \eta \norm{A^\tsp \Ply \bar G} \distA{\tilde x_k} \distA{\sbs{F}{c}(\tilde x_k,u_k)} 
\nonumber \\
&\quad 
+ \eta^2 \norm{\bar G^\tsp \Ply \bar G}
\distA{\sbs{F}{c}(\tilde x_k,u_k)}^2
+ 2 \norm{A^\tsp \Ply \bar H} \distA{\tilde x_k} \norm{\Delta w_k}
\nonumber \\
&\quad 
+ \norm{\bar H^\tsp \Ply \bar H} \norm{\Delta w_k}^2 
+ 2 \eta \norm{\bar G^\tsp \Ply \bar H} 
\distA{\sbs{F}{c}(\tilde x_k,u_k)} \norm{\Delta w_k}
\end{align*}
By completing the squares:
\begin{small}
\begin{align}
\label{eq:Wdot}
&W(\tilde x_{k+1}) - W(\tilde x_k ) \leq 
- (1-\epsilon) \underline \lambda(\Qly) \distA{\tilde x_k}^2 
\nonumber\\ 
& \quad 
+ \eta^2 \underbrace{
\left( \frac{2 \norm{A^\tsp \Ply \bar G}^2}{\epsilon 
\underline \lambda(\Qly)}  
+ \norm{\bar G^\tsp \Ply \bar G}\right) 
}_{:=b_1}
\distA{\sbs{F}{c}(\tilde x_k,u_k)}^2
\nonumber\\ 
& \quad 
+ \underbrace{
\left( \frac{2 \norm{A^\tsp \Ply \bar H}^2}{\epsilon 
\underline \lambda(\Qly)}  
+ \norm{\bar H^\tsp \Ply \bar H}\right) 
}_{:=b_2}
\norm{\Delta w_k}^2
\nonumber\\ 
& \quad 
+ 2 \eta \norm{\bar G^\tsp \Ply \bar H} 
\distA{\sbs{F}{c}(\tilde x_k,u_k)}^2
\norm{\Delta w_k}^2 .
\end{align}
\end{small}
\vspace{-.1cm}

\noindent
By combining \eqref{aux:Vdot-2}-\eqref{eq:Wdot} and by completing the
squares:
\begin{align*}
& U(\tilde  x_{k+1},u_{k+1})-U(\tilde x_{k},u_{k}) \leq 
-(1-\epsilon - \frac{\ell \eta}{2}) 
\distA{\sbs{F}{c}(\tilde x_k,u)}^2  
\nonumber\\
& \quad 
+ a_1 \distA{\tilde x_k}^2
+ a_2 \norm{u_{k+1}^*-u_k^*}^2
- (1-\epsilon) \underline \lambda(\Qly) \distA{\tilde x_k}^2 
\nonumber\\
& \quad 
+  \eta b_1 \distA{\sbs{F}{c}(\tilde x_k,u_k)}^2
+ (b_2 + b_3)  \norm{\Delta w_k}^2,
\end{align*}
where 
$b_3 = 2 \eta^2 \norm{\bar G^\tsp \Ply \bar H} /\epsilon$.
In summary, by letting 
$z :=  (\distA{\sbs{F}{c}(\tilde x_k,u)}, \distA{\tilde x_k}) $,
$v_1 := \norm{u_{k+1}^*-u_k^*}$, and $v_2 := \norm{\Delta w_k}$, 
if the following conditions are satisfied:
\begin{align*}
\eta &< \frac{1-\epsilon}{\ell/2 + b_1}, &
\underline{\lambda}(\Qly) &> \frac{a_1}{1-\epsilon},
\end{align*}
then $U(x_k,u_k)$ satisfies \eqref{eq:ISS-b} with:

\begin{align*}
\alpha_3(z) &= \min \{ (1-\epsilon) - \frac{\eta \ell}{2} - \eta b_1,
(1-\epsilon) \underline \lambda(\Qly) -a_1\} z^2,\\
\sigma(v_1,v_2) &=  a_2 v_1^2 + (b_2+b_3) v_2^2 .
\end{align*}
Finally, the claim follows by taking the supremum among all 
$u_k^*$ and by replacing the Hausdorff distance.
\end{pfof}

Theorem~\ref{thm:convergence} asserts that a sufficiently-small choice
of the stepsize $\eta$ is guarantees convergence to the optimizers, up
to an asymptotic error that depends on the time-variability of the 
optimal set and on the time-variability of the unknown exogenous 
signal~$w_k$. 
Although an exact computation of $\eta^*$ requires the
knowledge of the system matrices $(A, C)$, 
Theorem~\ref{thm:convergence} provides an existence claim for the 
stepsize $\eta$. Further, we note that the requirement
$\underline \lambda (\Qly) > \frac{a}{\epsilon ( 1 -\epsilon)}$
is non-restrictive, since the choice of $\Qly$ is arbitrary in 
Assumption~\ref{ass:stabilityPlant}.


\begin{remark}{\bf\textit{(Relationship with Classical Convergence Results)}}
We note that, when the plant is infinitely fast (i.e. the 
plant dynamics in \eqref{eq:gradientFlowController} are replaced by 
$y_k = G u_k + H w_k$), standard results 
(see e.g. \cite{SB-LV:04}) guarantee convergence of gradient-like 
dynamics for all $\eta< \sbs{\eta}{static}:= 2/\ell$.
By noting that $\eta^*$ in Theorem~\ref{thm:convergence} is strictly 
smaller that $\sbs{\eta}{static}$, the theorem suggests that a 
strictly smaller stepsize is required when the system dynamics are 
non-negligible.~
\QEDB
\end{remark}

\begin{figure}[t]
\centering \subfigure[]{\includegraphics[width=.48\columnwidth]{%
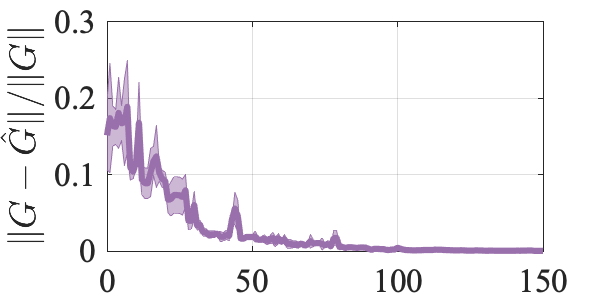}}
\centering \subfigure[]{\includegraphics[width=.48\columnwidth]{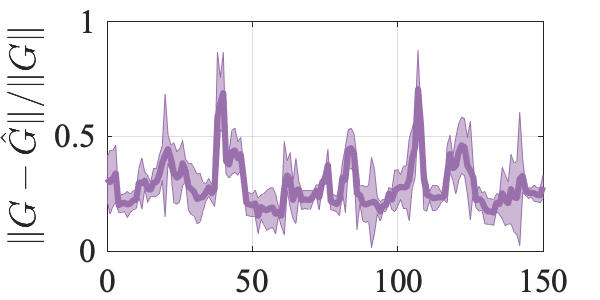}}
\centering \subfigure[]{\includegraphics[width=.48\columnwidth]{%
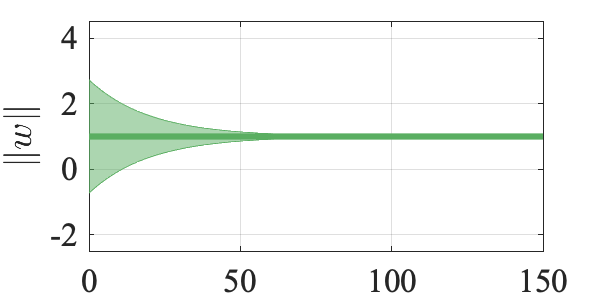}} 
\centering \subfigure[]{\includegraphics[width=.48\columnwidth]{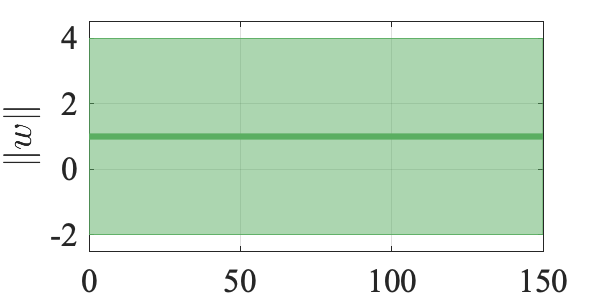}}
\caption{Error between data-driven $G$ and model-based $G$ when the 
training data is affected by an unknown disturbance $w_k$. 
The illustration has been produced by using the output of a Montecarlo
simulation where the realization of $w$ was varied over $10,000$ 
samples of a IID Gaussian process. Continuous lines illustrate the 
mean of the trajectory, shaded areas illustrates the 
$3$-standard deviation confidence intervals.
(a)-(c) When the exogenous disturbance is constant over time 
(equivalently, when the variance of $w_k$ in the Montecarlo 
simulations decays to zero asymptotically) then 
Corollary~\ref{cor:constantNoiseG} guarantees 
$\norm{G-\hat G} \rightarrow 0$. (b)-(d) When the training data is 
affected by time-varying disturbance (equivalently, when the variance 
of $w_k$ in the Montecarlo simulations is constant over time), 
Theorem~\ref{thm:data-drivenTransferFunction} allows us to approximate
$G$ up to a finite error.}
\label{fig:computeG}
\end{figure}

\section{Simulation Results}
\label{sec:6}
To illustrate the conclusions drawn  in Theorem 
\ref{thm:data-drivenTransferFunction}, we consider a system 
with $n = 20$, $m = r = 10$, $p=n$, and matrices $(A,B,C,E)$ with 
random entries that satisfy Assumption \ref{ass:stabilityPlant}.
Fig. \ref{fig:computeG} illustrates the error in the computed 
transfer function $\norm{G-\hat G}$, where $G$ denotes  the model-
based steady-state transfer function, and $\hat G$ denotes the 
transfer function computed according to 
Corollary~\ref{cor:constantNoiseG}.
$\hat G$ has been computed by using a rolling-horizon window that 
discards old samples over time. The length of the collection window is 
chosen equal to the smallest number of samples needed to guarantee 
persistence of excitation, and all samples (including the ones used 
for initialization) are obtain by using a random, persistently 
exciting, input.
Fig.~\ref{fig:computeG}(a)-(c) validates the conclusions drawn in 
Theorem~\ref{thm:data-drivenTransferFunction}: it shows that 
when the disturbance affecting the training data $w_k$ is constant, 
then the technique proposed in Corollary~\ref{cor:constantNoiseG}
allows us to compute $G$ from a sample trajectory.
Fig.~\ref{fig:computeG} illustrates numerically that the error 
$\norm{G-\hat G}$ is an increasing function of the 
time-variability of the disturbance $w_k$ affecting the training data.

Fig. \ref{fig:controlTrajectories}(a) illustrates the tracking error
for a simulation of the dynamics \eqref{eq:gradientFlowController}, 
subject to the time-varying disturbance $w$ illustrated in 
Fig. \ref{fig:controlTrajectories}(b). 
In this case, the map $G$ was computed from noiseless historical 
data, according to Theorem~\ref{thm:convergence}.
For simplicity, we consider the case where $
\phi(u) = u^\tsp Q_u u$, $Q_u \succ 0$ and 
$\psi(y)=(y-\sps{y}{ref})^\tsp Q_y (y-\sps{y}{ref})$, 
$Q_y \succ 0$.
The figure validates Theorem \ref{thm:convergence} by showing that the
tracking error is governed by two terms: (i) an error component 
associated with the initial conditions that decays to zero 
asymptotically, and  (ii) an error component associated with the 
time-variability of $w_k$, which vanishes only if $w_{k+1}-w_k = 0$.
The numerical simulations also suggest that values of $\eta$ larger 
than $\eta^*$  prevent the convergence of the method.

\begin{figure}[t]
\centering \subfigure[]{\includegraphics[width=\columnwidth]{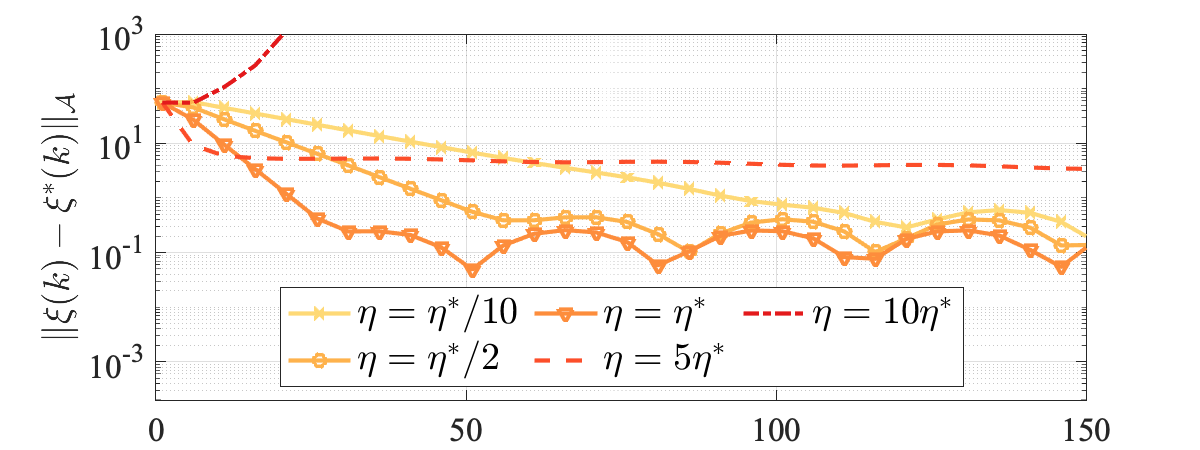}}
\centering \subfigure[]{\includegraphics[width=\columnwidth]{%
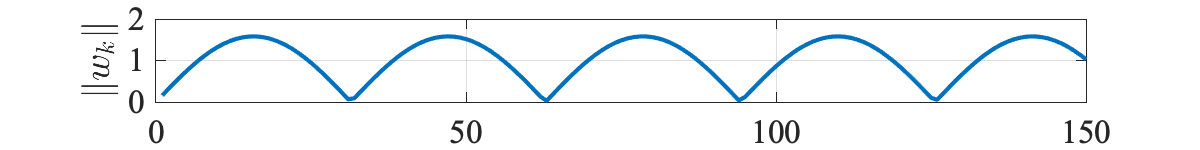}}
\caption{(a) Tracking error of \eqref{eq:gradientFlowController}. 
(b) Time-varying disturbance. The simulation suggests that values of 
$\eta$ larger than $\eta^*$ prevent convergence.
}
\label{fig:controlTrajectories}
\end{figure}

\section{Conclusions}
\label{sec:7}
We proposed a data-driven method to steer a dynamical system to the 
solution trajectory of a time-varying optimization problem.
The technique does not rely on any prior knowledge or estimation of the
system matrices or of the exogenous disturbances affecting the model 
equation. Instead, we showed that noise-free input-output data 
originated by the open-loop system can be used to compute the 
steady-state transfer function of the dynamical environment.
Moreover, we showed that convergence of the proposed algorithm to the 
time-varying optimizers is guaranteed when the dynamics of the 
controller are sufficiently slower than those of the dynamical system.
This work sets out several opportunities for future works, including 
extensions to scenarios where historical data is affected by 
non-constant noise terms, and the development of data-driven methods 
for the convergence analysis of the interconnected system.

\bibliographystyle{IEEEtran}
\bibliography{bib/brevalias,bib/main_GB,bib/Main,bib/GB,bib/add}
\end{document}